\documentclass[12pt]{amsart}

\usepackage[english]{babel}

\usepackage[a4paper]{geometry}
\usepackage{mathalfa}
\usepackage{tikz} 
\usepackage{caption}
\usepackage{subcaption}
\usetikzlibrary{patterns}
\usepackage{amsmath,amssymb,amsfonts,amsthm,a4,fleqn,bbm,dsfont,enumerate}
\usepackage{enumitem,cite}

\newtheorem{definition}{Definition}[section] 
\newtheorem{lemma}[definition]{Lemma}
\newtheorem{example}[definition]{Example}
\newtheorem{theorem}[definition]{Theorem}
\newtheorem{proposition}[definition]{Proposition}
\newtheorem{corollary}[definition]{Corollary}
\newtheorem{remark}[definition]{Remark}

\usepackage{graphicx}
\graphicspath{C:\Users\akhil\OneDrive\Desktop\Library Mathematics\Other}
\usepackage[colorlinks=true, allcolors=blue]{hyperref}
\makeatletter
\newcommand{\ostar}{\mathbin{\mathpalette\make@circled*}}
\newcommand{\make@circled}[2]{
	\ooalign{$\m@th#1\smallbigcirc{#1}$\cr\hidewidth$\m@th#1#2$\hidewidth\cr}
}
\newcommand{\smallbigcirc}[1]{
	\vcenter{\hbox{\scalebox{0.77778}{$\m@th#1\bigcirc$}}}
}
\@namedef{subjclassname@2020}{
	\textup{2020} Mathematics Subject Classification}
\makeatother

\begin{document}
	\title{Persistence diagrams of finite topological spaces}
	\author[A. Badra]{Akhilesh Badra}
	\address{AKHILESH BADRA; Department of Mathematics, University of Delhi,
		Delhi--110 007, India}
	\email{akhileshbadra028@gmail.com}
	\author[H. K. Singh]{Hemant kumar Singh}
	\address{HEMANT KUMAR SINGH; Department of Mathematics, University of Delhi,
		Delhi--110 007, India}
	\email{hemantksingh@maths.du.ac.in}
	\thanks{The first author is supported by research grant from the Council of Scientific and Industrial Research (CSIR), Ministry of Science and Technology, Government of India with reference number: 09/0045(13774)/2022-EMR-I}
	\thanks{The second author is supported by the IoE Grant of the University of Delhi  with Ref. No. IoE/2025- 26/12/FRP}


	\maketitle
	\begin{abstract}
		In this paper, we apply persistent homology to study finite topological spaces by associating persistence diagrams with their elements through the furtherness function. This approach enables us to classify elements according to their topological evolution within the space. We further establish a bottleneck stability theorem  showing that the bottleneck distance between the persistence diagrams associated with two elements is bounded above by their mutual furtherness.
	\end{abstract}

	\section{\textbf{\large{Introduction}}}
	
	In recent years, persistent homology has become an important tool in data analysis and has also contributed to new ways of studying metric spaces. In persistent homology, a persistence diagram records topological features of a geometric object at various scales as a multiset of points in the extended plane. These diagrams are constructed from a filtration of a topological space by recording when homological features appear and when they disappear \cite{Zomorodian}. In many applications, it is not enough to compute persistence diagrams, one also needs a way to measure how similar they are. A notion of similarity between persistence diagrams, known as the bottleneck distance, was introduced by Cohen-Steiner et al. in 2007 \cite{Cohen-Steiner}.\par
	Finite topological spaces have a natural combinatorial interpretation via preorders and posets. In 1937, Alexandroff established a fundamental connection between finite topological spaces and finite partially ordered sets\cite{Alexandroff}. In particular, finite $T_0$ spaces are in one-to-one correspondence with finite partially ordered sets. Later, in 1966, R. E. Stong and McCord used combinatorial methods to study the homotopy types of finite spaces and identified a connection between finite spaces and compact polyhedra, respectively \cite{McCord,Stong}. In our earlier work, we introduced the notion of furtherness on finite topological spaces, showing that every finite $T_0$ space can be viewed as an asymmetric metric space \cite{Badra}. We also defined a furtherness matrix to study the behavior of elements in such spaces. \par
	
	This work develops a new framework connecting finite topological spaces, order structures, and persistent homology through the notion of furtherness. Filtrations, persistence modules, and persistence diagrams are constructed for elements of finite spaces, providing a new way to study the topological evolutions  within finite spaces.\par
	
	Precisely, in this paper, we associate a persistence module and persistence diagram with each element of a finite $T_0$ topological space using a filtration induced by its furtherness function. We characterize the birth and death of persistence classes in terms of the minimal and cone points of the space, and use these characterizations to classify elements according to the similarity of their persistence diagrams. This classification gives rise to a natural equivalence relation whose equivalence classes exhibit  order theoretic properties and lead to a quotient space encoding all distinct topological evolutions within the space. Furthermore, we establish a bottleneck stability result showing that the associated persistence diagrams are stable under perturbation of the furtherness function, and extend the entire framework from finite $T_0$ spaces to arbitrary finite spaces through furtherness-preserving maps.
	\par

	The paper is organized as follows. In Section \ref{2}, we recall the necessary preliminaries on persistent homology and furtherness in finite topological spaces. In Section \ref{3}, we construct a filtration associated with each element $x$ of a finite $T_0$ space $X$ and define the corresponding persistence module $V_x$. Section \ref{4} is devoted to describe the structure of the associated persistence diagram $dgm(V_x)$. Section \ref{5} develops an equivalence relation determined by the persistence diagrams of elements of a finite $T_0$ space. This relation leads to a classification of elements and reveals structural properties of the induced equivalence classes. In Section \ref{6}, we extend the constructions and results obtained for finite $T_0$ spaces to arbitrary finite spaces. Finally, in Section \ref{7}, we establish a stability result by relating the bottleneck distance between the persistence diagrams associated with two elements to their mutual furtherness.\par 
	Throughout this paper, the term finite space refers to a finite topological space. Likewise, a finite $T_0$ space means a finite topological space satisfying the $T_0$ separation axiom.

	\section{\textbf{\large{Preliminaries}}}\label{2}
	In this section, we review the basics of persistent homology and furtherness in finite topological spaces.
	
	Persistent homology studies the evolution of topological features in a parameterized family of spaces using algebraic methods. We recall the necessary definitions and results of persistent homology: For details, see \cite{Carlson},  \cite{Chazal}, \cite{Dey} and \cite{Edelsbrunner}.

	Let $(I,\leq)$ be a totally ordered set, regarded as a category whose objects are the elements of $I$, with a unique morphism $i\rightarrow j$ whenever $i\leq j$. Let $\textbf{Vect}_{\mathbb{K}}$ denote the category of vector spaces over a field $\mathbb{K}$.
	
	\begin{definition}
		A persistence module indexed by $I$ is a covariant functor $ V : I \longrightarrow \textbf{Vect}_{\mathbb{K}}$.
	\end{definition}
	
	Explicitly, this consists of a family of vector spaces $\{V_i\}_{i\in I}$ together with linear maps $v_{i,j} : V_i \rightarrow V_j, ~\forall~ i \leq  j,$ such that $v_{i,i} =$ id and $v_{j,k} \circ v_{i,j} = v_{i,k}$ whenever $i \leq j \leq k$.\par 

	Persistent homology studies persistence modules that arise from filtrations via homology \cite{Oudot}. Let $\textbf{Top}$ denote the category of topological spaces and continuous maps. A \textit{filtration} indexed by $I$ is a functor $\mathcal{F} : I \longrightarrow \textbf{Top}$.
	A filtration $\mathcal{F}$ together with the homology functor $H_n(-,\mathbb{K})$ naturally gives rise to a persistence module.

	Given a simplicial complex $K$ and a monotonic function $f: K\rightarrow \mathbb{R}$, one obtains a filtration of $K$. By monotonicity, we mean that $f$ is non-decreasing along increasing chains of faces; that is, $f(\sigma)\leq f(\tau)$ whenever $\sigma$ is a face of $\tau$. This ensures that, for each $a\in \mathbb{R}$, the sublevel set $K_a = f^{-1}(-\infty,a]$ is a subcomplex of $K$. In this way, we obtain an increasing sequence $\emptyset \subseteq 	K_0 \subseteq K_1 \subseteq \cdots \subseteq K_n =K$, which forms a filtration $\mathcal{F}_f$ of $K$.

	\begin{definition}[\!\!\cite{Dey}]
		The persistence diagram $dgm_p(\mathcal{F}_f)$ of a filtration $\mathcal{F}_f$ induced by a function $f$ is the multiset
		of points $(a, b)\in \overline{\mathbb{R}}^2$ (extended real plane), each taken with its corresponding multiplicity, representing the birth time $``a"$ and death time $``b"$ of a $p$-dimensional homology class in the associated persistence module $V$. By convention, every point on the diagonal $\Delta=\{(t,t) ~|~ t\in \overline{\mathbb{R}}\}$ is included with infinite multiplicity.
	\end{definition}

	 Note that the (index) persistence of a class  $(a, b)\in dgm_p(\mathcal{F}_f)$ is $b-a$, which measures the lifetime of the corresponding homological feature. Furthermore, the elder rule states that whenever two classes merge, the older class persists while the younger class dies \cite{Edelsbrunner}. 
	
	\begin{definition}[\!\!\cite{Dey}]
		Let $X$ and $Y$ be two persistence diagram and $\Pi$ denote the set of all bijections $\pi:X\rightarrow Y$.  Consider the distance between the two points $x=(x_1,x_2)$ and $y=(y_1,y_2)$ in $L_\infty$-norm $||x-y||_\infty=\max\{|x_1-y_1|,|x_2-y_2|\}$ with the assumption that $\infty-\infty=0.$ The bottleneck distance between two diagrams $X$ and $Y$ is $d_B(X,Y)=\inf\limits_{\pi\in\Pi}\sup\limits_{x\in X}||x-\pi(x)||_\infty$. 
	\end{definition}
	
	 \begin{proposition}[Stability Theorem For Filtrations \cite{Edelsbrunner}]\label{S}
		{\normalfont
			Let $K$ be a simplicial Complex and $f,g:K\rightarrow \mathbb{R}$ two monotonic functions. For each dimension $p$, the bottleneck distance between the persistence diagram $X=Dgm_p(\mathcal{F}_f)$ and $Y=Dgm_p(\mathcal{F}_g)$ is bounded from above by the $L_{\infty}$-distance between the functions, $d_B(X,Y)\leq ||f-g||_{\infty}$.}
	\end{proposition}
	
	 A considerable amount of work by Alexandroff \cite{Alexandroff}, Stong \cite{Stong}, and McCord \cite{McCord} has shaped our understanding of finite topological spaces through their connections with posets and their homotopy  properties. We recall the basic constructions and results for finite topological spaces: For details, see\cite{Barmak}.\par
	
	Let $X$ be a finite topological space. For each $x\in X$, define
	$U_x := \bigcap\{U \subseteq X | U \text{ is open and } x \in U \}$.
	The collection $\{U_x | x \in X\}$ forms a minimal basis for the topology on X. We have the relation $\leq$ on $X$, defined by  $x \leq y \iff x \in U_y$  is a preorder. Conversely, given a finite preordered set $(X, \leq)$,  the collection of set $U_x := \{ y \in X | y \leq x \}, x \in X$, is a basis for a topology on $X$. This establishes an equivalence between finite topological spaces
	and finite preorders.\par 
	Given a preorder $(X, \leq)$, define an equivalence relation by $x \sim_0 y \iff x \leq y$ and $y \leq  x$.
	The quotient $X/$$\sim_0$ inherits a natural topology whose associated preorder is a partial order, and is homotopically equivalent to $X$.	
	In particular, every finite topological space is homotopy equivalent to a finite $T_0$ space, and the collection of finite $T_0$ spaces has one-to-one correspondence with the collection of  finite posets.  Throughout this paper, a minimal point of a finite $T_0$ space $X$ refers to a minimal element in its associated poset.
	
	\begin{definition}[\!\!\cite{Barmak}]	Let $X$ be a finite $T_0$ space. The order complex $\mathcal{S}(X)$ of $X$ is the simplicial complex whose simplices are the nonempty chains of $X$.
	\end{definition}
	
	If two points are comparable in $X$, then there is an edge between their corresponding vertices in the order comples $S(X)$.

	Now, we discuss furtherness in finite topological spaces introduced in our earlier work \cite{Badra}.
	Let $X$ be a finite space. A nested sequence of open sets around $x\in X$ is a sequence $(U_j)_{j\geq 0}$ of open subsets $U_j$ of $X$ such that  $U_0=U_x$ and $U_j\subset U_{j+1},$ and there is no open set between $U_j$ and $U_{j+1}$ for all $j\geq 0$, where $U_j\subset U_{j+1}$ means $U_j$ is properly contained in $U_{j+1}$.
	
	\begin{definition}[\!\!\cite{Badra}]
		Let $X$ be a finite space and $x,y\in X$. The furtherness $\Psi(x,y)$ of $x$ from $y$ is the smallest whole number $k$ such that there exist a nested sequence $(U_j)_{j\geq0}$ of open sets around $x$ such that $y\in U_k$.\par 
	\end{definition}

	\begin{remark}[\!\!\cite{Badra}]\label{2.10}
		{\normalfont Let $X$ be a finite $T_0$ space. Then, $\Psi(x,y)=|U_y\backslash U_x|, \forall x,y\in X.$}
	\end{remark}

	Note that the above relation $\sim_0$ on a finite space $X$ can also be viewed using furtherness function as $x\sim_0 y$ if $\Psi(x,y)=\Psi(y,x)=0$. Then, the quotient $X/$$\sim$ is $T_0$ and $\Psi_{X/\sim}([x],[y])=\Psi_{X}(x,y)$ for all $x,y\in X$. The furtherness matrix  $\Psi(X)$ of a finite space $X$  with $n$ elements $\{a_i|1\leq i\leq n\}$, is an $n\times n$ matrix, where $(i,j)^{th}$ entry of furtherness matrix is $\Psi(a_i,a_j)$, for $a_i,a_j\in X$. In particular, if $X$ if $T_0$, then $\Psi(X)_{i,j}=|U_{a_{j}}\backslash U_{a_{i}}|, \forall i,j\in X$.\par

	\section{\textbf{\large{Filtration induced by Furtherness}}}\label{3}
	
	In this section, we construct a filtration corresponding to each element of a finite $T_0$ space using its furtherness within the space.

	Let $X$ be a finite $T_0$ space. So, $X$ can be seen as a finite poset. We also have an associate order complex $\mathcal{S}(X)$ with it. Let $x\in X$ be an arbitrary element. We can construct a filtration of $\mathcal{S}(X)$ using the furtherness function of $X$ around $x$.\par 
	Let $X_x^k$ be the collection of all those points of $X$ such that furtherness of $x$ from that point is less than or equals to $k$, where $x\in X$, and $k$ is a non-negetive integer. Notice that, for every $x\in X$ and $k\geq 0$, $X_x^k$ is a subspace of $X$ and $\mathcal{S}(X_x^k)$ is a subcomplex of $\mathcal{S}(X)$. Also, for $k_1\leq k_2$, we have $X_x^{k_{1}} \subseteq X_x^{k_{2}}$, that implies $\mathcal{S}(X_x^{k_{1}})\subseteq \mathcal{S}(X_x^{k_{2}}), \forall x\in X$. Thus, for every $x\in X$, we have a filtration $\{X_x^k ~|~ 0\leq k\leq |X|\}$  of $X$, that starts from $U_x$, and a filtration $\{\mathcal{S}(X_x^k) ~|~ 0\leq k\leq |X|\}$  of $\mathcal{S}(X)$, that starts from $\mathcal{S}(U_x)$. So, corresponding to each element of $X$, we have a filtration of its order complex $\mathcal{S}(X)$. We refer $\mathcal{S}(X_x^k)$ as  the level $k$ of the filtation corresponding to $x$.\par 
	
	We describe this filtration through a monotonic function from the order complex of $X$ to $\mathbb{R}.$ For $x\in X$, consider a function $f_x:\mathcal{S}(X)\rightarrow\mathbb{R}$ such that $f_x(\langle a_0,a_1,...,a_n\rangle )=\max\{\Psi(x,a_i)~|~ 0\leq i\leq n\},$ and $\langle a_0,a_1,...,a_n\rangle
	\in \mathcal{S}(X).$ Let $\langle b_0,b_1,...,b_m\rangle$ be a subcomplex of $\langle a_0,a_1,...,a_n\rangle$. Then $b_j\in \{a_0,a_1,...,a_n\},
	 0\leq j \leq m \implies \Psi(x,b_j)\leq \max\{\Psi(x,a_i)~|~ 0\leq i\leq n\},  0\leq j \leq m \implies f_x(\langle b_0,b_1,...,b_n\rangle )\leq f_x(\langle a_0,a_1,...,a_n\rangle ).$ So, $f_x$ is monotonic. By the  monotonicity of $f_x$, we get that the sublevel set $\mathcal{S}(X)_x^k= f_x^{-1}(-\infty,k]$ is a subcomplex of $\mathcal{S}(X)$. Therefore, we can arrange it in an increasing sequence $\emptyset\subseteq \mathcal{S}(X)_x^0 \subseteq \mathcal{S}(X)_x^1\subseteq ... \subseteq \mathcal{S}(X)_x^n =\mathcal{S}(X)$, for some $n\leq |X|-1$. This sequence of complexes is a filtration of $\mathcal{S}(X)$ obtained from $f_x$.
	Next, we observe that $\mathcal{S}(X_x^k) = \mathcal{S}(X)_x^k, \forall x\in X, k\geq 0.$ From our construction, $\mathcal{S}(X_x^k)$ is the order complex of $X_x^k=\{y\in X|\Psi(x,y)\leq k\}$ and $\mathcal{S}(X)_x^k= f^{-1}(-\infty,k]$ is the subcomples of $\mathcal{S}(X)$. Consider, $a_0\leq a_1\leq ... \leq a_n$ such that $\langle a_0, a_1, ... , a_n\rangle \in \mathcal{S}(X)_x^k$. Note that $f_x(\langle a_0, a_1, ... , a_n\rangle)\leq k \iff \Psi(x,a_i)\leq k \iff a_i\in X_x^k,~ 0\leq i \leq n \iff \langle a_0, a_1, ... , a_n\rangle \in \mathcal{S}(X_x^k)$. So, $\mathcal{S}(X_x^k) = \mathcal{S}(X)_x^k, \forall x\in X, k\geq 0.$ 
	Therefore, the filtration obtained from $f_x$ coincides with the filtration constructed earlier.\par 
	
	So, for every $x\in X$, its associated filtration is a functor $F_x:\mathbb{R}\rightarrow \textbf{Top}$ such that $F_x(a)=  |f_x^{-1}(-\infty,a]|= |\mathcal{S}(X)_x^a|$, where $|K|$ means the geometric carrier of a simplicial complex K.
	 The functoriality of homology gives a persistence module. Thus, for every $x\in X$,  we have an associated persistence module $V_x: \mathbb{R} \rightarrow \mathbf{Vect}_{\mathbb{K}}$ such that $V_x(a)=H_0(F_x(a),\mathbb{K})$, where $H_0(F_x(a),\mathbb{K})$ is the $0^{th}$ homology vector space of topological space $F_x(a)$ with coefficients in a field $\mathbb{K}$.\par 
	 
	 	We now illustrate the filtrations associated with the elements of a space through the following example.

	 	\begin{center}
	 	\begin{tikzpicture}  
	 		[scale=.9,auto=center,every node/.style={circle,fill=blue!10}] 
	 		
	 		\node (r) at (0,1) [ball color=cyan, circle, draw=black, inner sep=0.1cm, label=left:$r$] {};  
	 		\node (s) at (1,0)  [ball color=cyan, circle, draw=black, inner sep=0.1cm, label=left:$s$]{}; 
	 		\node (p) at (2,1) [ball color=cyan, circle, draw=black, inner sep=0.1cm, label=below :$p$]{};
	 		\node (q) at (3,2) [ball color=cyan, circle, draw=black, inner sep=0.1cm, label=below :$q$]{};  
	 		\node (t) at (4,1) [ball color=cyan, circle, draw=black, inner sep=0.1cm, label=below :$t$]{};  
	 		\node (u) at (5,1.5) [ball color=cyan, circle, draw=black, inner sep=0.1cm, label=below :$u$]{};

	 		\draw (r) -- (s);
	 		\draw (s) -- (p); 
	 		\draw (p) -- (q);
	 		\draw (q) -- (t);
	 		\draw (u) -- (t);

	 		\node[ball color=black, circle, draw=black, inner sep=1pt,label=below:$s$] (s) at (8,0) {};
	 		\node[ball color=black, circle, draw=black, inner sep=1pt, label=below:$p$] (p) at (10,0) {};
	 		\node[ball color=black, circle, draw=black, inner sep=1pt,label=above:$q$] (q) at (9,1.732) {};
	 		\node[ball color=black, circle, draw=black, inner sep=1pt,label=above:$r$] (r) at (7,1) {};
	 		\node[ball color=black, circle, draw=black, inner sep=1pt,label=above:$t$] (t) at (10,1) {};
	 		\node[ball color=black, circle, draw=black, inner sep=1pt,label=above:$u$] (u) at (11,1.732) {};

	 		\draw (s) -- (p) -- (q) -- (s);
	 		\draw (r) -- (s);
	 		\draw (q) -- (t);
	 		\draw (u) -- (t);

	 		\fill[pattern=north east lines, pattern color=black] (8,0) -- (10,0) -- (9,1.732) -- cycle;
	 	
	 	\end{tikzpicture}  
	 \end{center}
	 \hspace{40mm}	$X$ \hspace{15mm} Example \ref{example1} \hspace{20mm} $\mathcal{S}(X)$\\

	\begin{example}\label{example1}
		{\normalfont
			Let $X=\{p,q,r,s,t,u\}$ be a finite poset, whose Hasse diagram is shown in the figure above. Then $X$ is a finite $T_0$ space.\par 
			By Remark \ref{2.10}, for $s\in X$, we have  $X_s^0=\{s\}, X_s^1=\{s,r,p,t\}, X_s^2=\{s,r,p,t,u\},\\
			 X_s^3=X,$ and fot $t\in X$, we have $X_t^0=\{t\}, X_t^1=\{s,t,u\}, X_t^2=\{s,r,p,t,u\}$, $X_t^3=X.$

			\begin{center}
				\begin{tikzpicture}  
					[scale=.9,auto=center,every node/.style={circle,fill=blue!10}]
					\node (a) at (1.5,0)[ball color=black, circle, draw=black, inner sep=1pt,label=below:$s$]{};

					\node (r) at (3,1) [ball color=black, circle, draw=black, inner sep=1pt,label=below:$r$] {};  
					\node (s) at (4,0) [ball color=black, circle, draw=black, inner sep=1pt,label=below:$s$]{}; 
					\node (p) at (5,1) [ball color=black, circle, draw=black, inner sep=1pt,label=below:$p$]{}; 
					\node (t) at (6,1) [ball color=black, circle, draw=black, inner sep=1pt,label=below:$t$]{};

					\draw (r) -- (s);
					\draw (s) -- (p); 
					
					\node (1) at (8,1) [ball color=black, circle, draw=black, inner sep=1pt,label=below:$r$] {};  
					\node (2) at (9,0)  [ball color=black, circle, draw=black, inner sep=1pt,label=below:$s$]{}; 
					\node (3) at (10,1) [ball color=black, circle, draw=black, inner sep=1pt,label=below:$p$]{}; 
					\node (4) at (11,1) [ball color=black, circle, draw=black, inner sep=1pt,label=below:$t$]{};  
					\node (5) at (11.5,1.5) [ball color=black, circle, draw=black, inner sep=1pt,label=below:$u$]{};  
					
					\draw (1) -- (2);
					\draw (2) -- (3); 
					\draw (4) -- (5);

					\node[ball color=black, circle, draw=black, inner sep=1pt,label=below:$s$] (s') at (14,0)  {};
					\node[ball color=black, circle, draw=black, inner sep=1pt, label=below:$p$] (p') at (16,0) {};
					\node[ball color=black, circle, draw=black, inner sep=1pt,label=above:$q$] (q') at (15,2) {};
					\node[ball color=black, circle, draw=black, inner sep=1pt,label=above:$r$](r') at (13.5,1) {};
					\node[ball color=black, circle, draw=black, inner sep=1pt,label=above:$t$] (t') at (16,1) {};
					\node[ball color=black, circle, draw=black, inner sep=1pt,label=above:$u$] (u') at (16.5,1.5) {};

					\draw (r') -- (s');
					\draw (s') -- (p'); 
					\draw (s') -- (q');
					\draw (p') -- (q');
					\draw (q') -- (t');
					\draw (u') -- (t');

					\fill[pattern=north east lines, pattern color=black] (14,0) -- (15,2) -- (16,0) -- cycle;

				\end{tikzpicture}

			\end{center}
			$\mathcal{S}(X_s^0)$ \hspace{20mm} $\mathcal{S}(X_s^1)$ \hspace{27mm} $\mathcal{S}(X_s^2)$ \hspace{25mm} $\mathcal{S}(X_s^3)=\mathcal{S}(X)$

			\begin{center}
				\begin{tikzpicture}  
					[scale=.9,auto=center,every node/.style={circle,fill=blue!10}] 
					\node (a) at (1.5,0)[ball color=black, circle, draw=black, inner sep=1pt,label=below:$t$]{};

					\node (s) at (3.5,0) [ball color=black, circle, draw=black, inner sep=1pt,label=below:$s$]{}; 
					\node (t) at (4.5,0.5) [ball color=black, circle, draw=black, inner sep=1pt,label=below:$t$]{};  
					\node (u) at (5.5,1) [ball color=black, circle, draw=black, inner sep=1pt,label=below:$u$]{};

					\draw (u) -- (t);

					\node (1) at (8,1) [ball color=black, circle, draw=black, inner sep=1pt,label=below:$r$] {};  
					\node (2) at (9,0)  [ball color=black, circle, draw=black, inner sep=1pt,label=below:$s$]{}; 
					\node (3) at (10,1) [ball color=black, circle, draw=black, inner sep=1pt,label=below:$p$]{}; 
					\node (4) at (11,1) [ball color=black, circle, draw=black, inner sep=1pt,label=below:$t$]{};  
					\node (5) at (11.5,1.5) [ball color=black, circle, draw=black, inner sep=1pt,label=below:$u$]{};  
					
					\draw (1) -- (2);
					\draw (2) -- (3); 
					\draw (4) -- (5);

					\node[ball color=black, circle, draw=black, inner sep=1pt,label=below:$s$] (s') at (14,0)  {};
					\node[ball color=black, circle, draw=black, inner sep=1pt, label=below:$p$] (p') at (16,0) {};
					\node[ball color=black, circle, draw=black, inner sep=1pt,label=above:$q$] (q') at (15,2) {};
					\node[ball color=black, circle, draw=black, inner sep=1pt,label=above:$r$](r') at (13.5,1) {};
					\node[ball color=black, circle, draw=black, inner sep=1pt,label=above:$t$] (t') at (16,1) {};
					\node[ball color=black, circle, draw=black, inner sep=1pt,label=above:$u$] (u') at (16.5,1.5) {};

					\draw (r') -- (s');
					\draw (s') -- (p'); 
					\draw (s') -- (q');
					\draw (p') -- (q');
					\draw (q') -- (t');
					\draw (u') -- (t');

					\fill[pattern=north east lines, pattern color=black] (14,0) -- (15,2) -- (16,0) -- cycle;

				\end{tikzpicture}

			\end{center}
			
			$\mathcal{S}(X_t^0)$ \hspace{15mm} $\mathcal{S}(X_t^1)$ \hspace{30mm} $\mathcal{S}(X_t^2)$ \hspace{25mm} $\mathcal{S}(X_t^3)=\mathcal{S}(X)$\\

			Similarly, for $p\in X$, we have $ X_p^0=\{s,p\}, X_p^1=\{s,r,p,t\}, X_p^2=X,$ and for  $q\in X$, we have $ X_q^0=\{s,p,t,q\}, X_q^1=X.$
			
			\begin{center}
				\begin{tikzpicture}  
					[scale=.9,auto=center,every node/.style={circle,fill=blue!10}]
					\node (a) at (1,0)[ball color=black, circle, draw=black, inner sep=1pt,label=below:$s$]{}; 
					\node (b) at (1,1) [ball color=black, circle, draw=black, inner sep=1pt,label=above:$p$]{}; 
					
					\draw (a) -- (b);

					\node (r) at (4,1) [ball color=black, circle, draw=black, inner sep=1pt,label=below:$r$] {};  
					\node (s) at (5,0) [ball color=black, circle, draw=black, inner sep=1pt,label=below:$s$]{}; 
					\node (p) at (6,1) [ball color=black, circle, draw=black, inner sep=1pt,label=below:$p$]{}; 
					\node (t) at (7,1) [ball color=black, circle, draw=black, inner sep=1pt,label=below:$t$]{};

					\draw (r) -- (s);
					\draw (s) -- (p);

					\node[ball color=black, circle, draw=black, inner sep=1pt,label=below:$s$] (s') at (10,0)  {};
					\node[ball color=black, circle, draw=black, inner sep=1pt, label=below:$p$] (p') at (12,0) {};
					\node[ball color=black, circle, draw=black, inner sep=1pt,label=above:$q$] (q') at (11,2) {};
					\node[ball color=black, circle, draw=black, inner sep=1pt,label=above:$r$](r') at (9.5,1) {};
					\node[ball color=black, circle, draw=black, inner sep=1pt,label=above:$t$] (t') at (12,1) {};
					\node[ball color=black, circle, draw=black, inner sep=1pt,label=above:$u$] (u') at (12.5,1.5) {};

					\draw (r') -- (s');
					\draw (s') -- (p'); 
					\draw (s') -- (q');
					\draw (p') -- (q');
					\draw (q') -- (t');
					\draw (u') -- (t');

					\fill[pattern=north east lines, pattern color=black] (10,0) -- (11,2) -- (12,0) -- cycle;

				\end{tikzpicture}

			\end{center}

			\hspace{10mm} $\mathcal{S}(X_p^0)$ \hspace{30mm} $\mathcal{S}(X_p^1)$ \hspace{30mm} $\mathcal{S}(X_p^2)=\mathcal{S}(X)$

			\begin{center}
				\begin{tikzpicture}  
					[scale=.9,auto=center,every node/.style={circle,fill=blue!10}]

					\node (s) at (2,0) [ball color=black, circle, draw=black, inner sep=1pt,label=below:$s$]{}; 
					\node (p) at (4,0) [ball color=black, circle, draw=black, inner sep=1pt,label=below:$p$]{}; 
					\node (t) at (4,1) [ball color=black, circle, draw=black, inner sep=1pt,label=below:$t$]{};  
					\node (q) at (3,2) [ball color=black, circle, draw=black, inner sep=1pt,label=above:$q$] {};

					\draw (p) -- (q);
					\draw (s) -- (p); 
					\draw (t) -- (q); 
					\draw (s) -- (q); 
					
					\fill[pattern=north east lines, pattern color=black] (4,0) -- (3,2) -- (2,0) -- cycle;

					\node[ball color=black, circle, draw=black, inner sep=1pt,label=below:$s$] (s') at (8,0)  {};
					\node[ball color=black, circle, draw=black, inner sep=1pt, label=below:$p$] (p') at (10,0) {};
					\node[ball color=black, circle, draw=black, inner sep=1pt,label=above:$q$] (q') at (9,2) {};
					\node[ball color=black, circle, draw=black, inner sep=1pt,label=above:$r$](r') at (7.5,1) {};
					\node[ball color=black, circle, draw=black, inner sep=1pt,label=above:$t$] (t') at (10,1) {};
					\node[ball color=black, circle, draw=black, inner sep=1pt,label=above:$u$] (u') at (11,1.5) {};

					\draw (r') -- (s');
					\draw (s') -- (p'); 
					\draw (s') -- (q');
					\draw (p') -- (q');
					\draw (q') -- (t');
					\draw (u') -- (t');

					\fill[pattern=north east lines, pattern color=black] (10,0) -- (9,2) -- (8,0) -- cycle;

				\end{tikzpicture}

			\end{center}

			\hspace{30mm} $\mathcal{S}(X_q^0)$ \hspace{40mm} $\mathcal{S}(X_q^1)=\mathcal{S}(X)$

		}
	\end{example}

	In example \ref{example1}, we can notice that different points of a space $X$ may give different topological evolution of associated order complex $\mathcal{S}(X)$ with respect to furtherness. Filtration corresponding to   $s$ and $t$ had 3 different levels before completing the order complex of $X$, where as filtration of $p$ and $q$ has just 2 and 1 levels, respectively.\par
	
	We can also notice that during these filtrations some connected components appear, grow and merge into each other. As in the filtration corresponding to $s$, level 0 has just one connected component $\{s\}$, and at level 1 of this filtration a new component $\{t\}$ appears and at level 3, both of these component merge into each other.\par
	
 The persistence module $V_x$ associated to an element $x$ of a finite $T_0$ space $X$, can be used to measure the scale or resolution of a topological feature. Using connectedness as a feature in  this persistence module, we can observe  that when two elements have  the same topological evolution with respect to furtherness and what kind of similarities are their in this evolution.

	\section{\textbf{\large{Persistence diagrams of finite $T_0$ spaces}}}\label{4}
	
	In this section, we discuss results describing the structure of persistence diagram of persistence module associated with an element of finite $T_0$ space.\par 
	Let $X$ be a finite $T_0$ space. For each $x\in X$, we denote the persistence diagram of the persistence module $V_x$ by $dgm(V_x)$. In these persistence diagrams, classes are corresponding to the connected components of the space at different levels of filtration. So, corresponding to each connected components of $\mathcal{S}(X)$, we have a class of infinite persistence in $dgm(V_x), ~\forall~ x\in X$.
	
	The following result describes when a class appear in the persistence diagram
	of a persistence module $V_x$.
	
	\begin{theorem}\label{A}
		Let $X$ be a finite $T_0$ space and $x\in X$. Then
		\begin{itemize}
			\item[$(i)$] $[0,\infty)\in dgm(V_x)$ with multiplicity one.
			\item[$(ii)$] If $X$ has $m$ minimal points which are not comparable with $x$, then $m$ new classes appear at level 1.
			\item[$(iii)$] No new class appear after level 1.
		\end{itemize}
		
	\end{theorem}
	\begin{proof}
		Let $a$ be a minimal point of $X$. Then, either $a\leq x$ or $x$ is not comparable with $a$. Thus, by Remark \ref{2.10}, we get either $\Psi(x,a)=0$ or $\Psi(x,a)=1$.\par 
		First, suppose $\Psi(x,a)=0$. At level $0$, we know that  $\mathcal{S}(U_x)$ appears. A class corresponding to the connected component of $x$ appears in $dgm(V_x)$. As $U_x$ is connected, $\mathcal{S}(U_x)$ is connected. So, just one class $[0,-)$ appears at level 0 in $dgm(V_x)$. By elder rule, if any class merge with this class at any level, this class $[0,-)$ will go on, and merged class will end at that level. Therefore, this class will have infinite persistence. Hence $[0,\infty) \in dgm(V_x)$ and has multipicity 1.\par 
		
		If $\Psi(x,a)=1$, then $a$ is  not comparable with $x$. In this case, a new connected component will appear at level 1 corresponding to $a$. Note that $X_x^0= U_x$ is the collection all those points such that furtherness of $x$ from those points is 0. If $t\in U_x$ is comparable with $a$, then $\Psi(t,a)=0 \implies \Psi(x,a)\leq \Psi(a,t)+\Psi(t,a)=0 \implies x$ is comaparable with $a$, a contradiction. So, $X_x^0= U_x$ does not contain any point which is comparable with $a$. Thus, there is no vertex in $\mathcal{S}(X_x^0)$ that share an edge with $a$ in $\mathcal{S}(X)$. So, at level 1, when $a$ appears, it does not share any edge with the vertices of level 0. Therefore, it forms a new connected component at level 1. A new class $[1,n_a)$ appeares in $dgm(V_x)$ corresponding to $a$, where $n_a\in \mathbb{N}\cup \{\infty\}$ and $n_a>1$. Hence, if there are $m$ minmal points in $X$ not comparable with $x$,  then $m$ new classes appear at level 1.\par 
		After level 1, if any point $y$ appears, then $y$ is not a minimal point and it is comparable to at least one of the minimal points appeared in level 0 or 1. So, if $y$ appears in $X_x^t, t>1$, then $y$ appears with an edge joining $y$ with an earlier appeared vertex (corresponding to a minimal point) in $S(X_x^t)$. Thus, no new connected component created corresponding to $y$, and hence, no new class starts from $t$ in $dgm(V_x)$, for $t>1$. 
	\end{proof}
	
	 Above results describe the structure of the persistence diagram $dgm(V_x)$ and relate it to the minimal points of X. From above theorem we can notice that level 0 and level 1 of the filtration are the only levels at which a class can appear in $dgm(V_x)$. If a class appears at level 0 then it is of inifinite persistence. And if a class starts at level 1, it means it is appeared due to the existence of a minimal point of $X$ which is not comparable to $x$. Note that, in a space with just one minimal point, every persistence diagram has exactly one class $[0,\infty)$ and has no class of finite persistence.\par 
	
	From Theorem \ref{A}, it follows that for each minimal point of $X$ that is not comparable to $x$, there is a corresponding class in $dgm(V_x)$. In addition, all minimal points that are comparable to $x$ contribute to a single class in $dgm(V_x)$. These are all possible classes in $dgm(V_x)$. So, we have

	\begin{corollary}\label{B}
		Let $X$ be a finite $T_0$ space and $x\in X$. Then number of classes in the persistence diagram $dgm(V_x)$ is less than or equal to number of minimal points of $X$.
	\end{corollary}

	\begin{remark}\label{F}
		{\normalfont
			Let $|dgm(V_x)|$ denote the cardinality of the multiset $dgm(V_x)$, then by Theorem \ref{A}, we have $|dgm(V_x)|= m+1$, where $m$ is number of minimal points of finite $T_0$ space $X$ that are not comparable with $x$.
		}
	\end{remark}

	In example \ref{example1}, we can notice that, in $dgm(V_s)$, one class starts from level 0 corresponding to $U_s$ and one class starts from level 1 corresponding to $t$, which is the only minimal point of $X$ not compararable to $s$. We can also observe that at level 3, when $q$ enters $X_s^3$, both of the these classes merge to each other because $q$ is comparable with both $s$ and $t$. \par
	In the next result, we observe that when a point appears in $X_x^t$, which is comparable to two minimal points of different componets of previous level, and a class disappear in $dgm(V_x)$.

	\begin{theorem}\label{D}
		Let $X$ be a finite $T_0$ space and $x\in X$. Let $b$ be a minimal point of $X$ not comparable with $x$. If a class born at level 1 in $dgm(V_x)$ corresponding to $b$, dies at level $n_b<\infty,$ then there is an element in $X_x^{n_b}$ which is comparable with both $b$ and another minimal point of $X$. Moreover, if no such element exists then $n_b=\infty$.
	\end{theorem}

	\begin{proof}
	
	 Let $\{a_1,a_2,...,a_k\}$ be the collection of minimal points of $X$, which are not comparable to $x$. By Theorem \ref{A}, for each $a_i, 1\leq i\leq k$, we have a class $[1,n_{a_i})$ in $dgm(V_x)$, where $n_{a_i}\in \mathbb{N}\cup \{\infty\}$. Note that, we have a component in $S(X_x^1)$ corresponding to each $a_i$  and  one component containing $x$, $1\leq i\leq k$. So,  we have $k+1$ components at level 1.\par 
		
		Next, we observe that $n_{a_j}= \min \{ t ~|~ y\in X_x^t \text{ such that } y \text{ is comparable with both }$ $a_i \text{ and } a_j,\text{ for	some } 0\leq i< j\leq k\}, 1\leq j\leq k$, where $a_0$ is a minimal point comparable with $x$.\par 
		
		Notice that if a point $y$ enters  $X_x^2$ which is comparable with both $a_i \text{ and } a_j,\text{ for }\\
		 \text{ some } 0\leq i < j\leq k,$ then there will be an edge between $y$ and $a_i$ along with an edge between $y$ and $a_j$ in $\mathcal{S}(X_x^2)$. So, two components of $\mathcal{S}(X_x^1)$ merge into one component of $\mathcal{S}(X_x^2)$ due to the entrance of $y$. If both $i$ and $j$ are non-zero then one class  corresponding to $a_i$ or $a_j$ ends and one persists. So, in this case, WLOG, we may take  $[1,n_{a_j})=[1,2)\in dgm(V_x)$.
		Even if  $i=0$, then by elder rule the class corresponding to $a_0$ will persist and other will end. So, $[1,n_{a_j})=[1,2)\in dgm(V_x)$.
		If no such $y$ enters $X_x^2$, then no class ends at level 2. In this case $[1,2)\notin dgm(V_x)$.\par 
		
		We repeat this process untill a point $y$ appears in $X_x^t$, for some $t\geq 2$, where $y$ is comparable with both $a_i$ and $a_j$, for some $0\leq j \leq k$. So, we can say that the class corresponding to minimal point $a_j, 1\leq j\leq k,$ ends at $n_{a_{j}}$, where $n_{a_{j}}$ is the earliest level $t$ such that $y$  enters $X_x^t$, and $y$ is comparable with both $a_i$ and $a_j$, for some $0\leq i \leq k$.  If no such $y$ enters $X_x^t$, $\forall t\geq 1$, then $n_{a_j}=\infty.$
		\end{proof}
	
	In the above proof, one may notice that $n_{a_{j}}=\min\{\Psi(x,y)~|~ y \in X_x^{\Psi(x,y)} \text{ such that }\\
	 y \text{ is comparable with both $a_j$ and another minimal point of $X$}\}$.\par

	Notice that in Example \ref{example1}, In  $dgm(V_s)$ class started at level 1 corresponding to $t$ ends when $q$ enters in level 3. So, by Theorem \ref{D}, $[1,3)\in dgm(V_s)$. Therefore, $dgm(V_s)=\{[0,\infty),[1,3)\}$. Similarly, by Theorem \ref{A} and \ref{D}, we get  $dgm(V_t)=dgm(V_u)=dgm(V_r)=\{[0,\infty),[1,3)\}$, $dgm(V_p)=\{[0,\infty),[1,2)\}$, and $dgm(V_q)=\{[0,\infty)\}$.
	
	\begin{example}\label{example2}
	Let $X$ be a discrete space with $n$ points. Then $dgm(V_x)$ consists of $n$ classes corresponding to each point. In particular, $[0,\infty)\in dgm(V_x)$ with multiplicity 1 and $[1,\infty)\in dgm(V_x)$ with multiplicity $n-1$.
	\end{example}

	In the following result, we relate the cardinalities of persistence diagram corresponding to comparable elements.
	
	\begin{theorem}\label{K}
	    Let $X$ be a finite $T_0$ space and $x,y\in X$ such that $x\leq y$. Then $|dgm(V_x)|\geq |dgm(V_y)|$.
	\end{theorem}
	
	\begin{proof}
	    By Theorem \ref{A}, corresponding to all the  minimal points comparable with $x$, we have just one class $[0,\infty)\in dgm(V_x).$ As $x\leq y$,  minimal points which are comparable with $x$ are also comparable with $y$. So, again by Theorem \ref{A}, $[0,\infty)\in dgm(V_y).$\par 
	    Let $a\in X$ be a minimal point of $X$ that is not comparable with $y$. Thus, $a$ is not comparable with $x$. Again, by Theorem \ref{A}, corresponding to each minimal point $a\in X$ not comparable with $y$, we have a class $[1,n_a^y)\in dgm(V_y)$, and accordingly a class $[1,n_a^x)\in dgm(V_x)$. So, corresponding to each element of $dgm(V_y)$, we have an element of $dgm(V_x)$. Therefore, $|dgm(V_x)|\geq |dgm(V_y)|$.
	\end{proof}
	
	Notice that in above result $dgm(V_x)$ may have some classes corresponding to the minimal point of $X$ which are comparable with $y$ but not comparable with $x$. In Example \ref{example1}, we can see that $p \leq q$ and $t\in X$ is a minimal point comparable with $q$ but not with $p$, and we have  $2=|dgm(V_p)|> |dgm(V_q)|=1$.\par 
	
	From the proof of Theorem \ref{K}, we get that if $x\leq y$ then corresponding to every class $[1,n_a^x)$ in $dgm(V_x)$, we have a class $[1,n_a^y)$ in $dgm(V_y)$. In the next result, we observe a relationship between $n_a^x$ and $n_a^y$.

	\begin{theorem}\label{E}
	    Let $X$ be a finite $T_0$ space and $x,y\in X$ such that $x\leq y$. Let $a$ be minimal point of $X$ such that $a$ is not comparable with $y$ and $[1,n_a^x)$ $\&$ $[1,n_a^y)$ are classes corresponding to $a$  in $dgm(V_x)$ and $dgm(V_y)$, respectively. Then $n_a^y\leq n_a^x$.
	\end{theorem}
	
	\begin{proof}
	      As $a$ is a minimal point of $X$ not comparable with $y$, we get that $a$ is not comparable with $x$. By Theorem \ref{D}, corresponding to $a$, we get a class $[1,n_a^x)$  in $dgm(V_x)$. If $n_a^x<\infty$, then  a point $z$ enters $X_x^{n_a^x}$ such that $z$ is comparable with $a$ and another minimal point $X$. As $x\leq y$, we have $U_x\subseteq U_y \implies |U_z\backslash U_y| \leq  |U_z\backslash U_x| \implies \Psi(y,z)\leq \Psi(x,z).$ Therefore, if $n_a^x<\infty$, then we  must have $n_a^y<\infty$. We know that $n_a^y= \min\{\Psi(y,z')~|~  z' \in X_y^{\Psi(y,z')}  \text{ such that } z'\text{ is comparable with both $a_j$ and ano-}\\
	      	\text{ther minimal point of $X$}\}$. Hence, $n_a^y\leq n_a^x.$
	     \end{proof}

	In Example \ref{example1}, we can notice that $s\leq p$ and $t$ is a minimal point not comparable with $p$. We have, $2= n_t^p < n_t^s =3.$\par 
	We can generalize above result for arbitrary length of chain in a finite $T_0$ space.
	
	\begin{corollary}
		Let $X$ be a finite $T_0$ space and  $x_1\leq x_2\leq ... \leq x_k$ be a chain of elements of $X$. Let $``a"$ be a minimal point of $X$ such that $a$ is not comparable with $x_k$ and class $[1,n_a^{x_i})$ in $dgm(V_{x_i})$  corresponding to $``a"$ for $1\leq i\leq k$. Then $n_a^{x_1}\geq n_a^{x_2}\geq ... \geq n_a^{x_k}$.
	\end{corollary}
	
	\section{\textbf{\large{Classification via Persistence Diagrams}}}\label{5}
	
	In this section, we study an equivalence relation on finite $T_0$ spaces, and interpret it in terms of persistence diagrams. Using this relation, we classify elements, analyze the resulting quotient spaces, and examine the behavior of elements in this setting.\par 
	
	In the previous section, we observed that in a finite $T_0$ space X, minimal points give rise to the birth of classes in the persistence diagram associated with any element. We also saw that points comparable with more than one minimal point are responsible for the death of classes.\par 
	
	Let $X$ be a finite $T_0$ space and $x\in X$. We say that $x$ is a  \textit{cone point}, if $x$ is comparable to at least two minimal points and more than one edges are going downward from $x$ in the Hasse diagram of $X$.	In Example \ref{example1}, we have seen that $s$ and $t$ are minimal points and $q$ is a cone point of $X$.\\
	
	 \noindent
	\begin{minipage}{0.55\textwidth}
		 \begin{example}\label{example 4}{\normalfont
		 	Let $Y$ and $Z$ be finite $T_0$ spaces whose poset diagrams are shown on the right. Notice that $d$ is not a cone point of $Y$ even though it is comparable with the two minimal points $a$ and $b$. Similarly, $h$ is not a cone point of $Z$ even though two  edges going downward from it  in the poset diagram.}
		 \end{example}
	\end{minipage}
	\hfill
	\begin{minipage}{0.4\textwidth}
		
			\begin{center}
			\begin{tikzpicture}  
				[scale=.9,auto=center,every node/.style={circle,fill=blue!10}] 
				
				\node (a) at (1.5,0) [ball color=cyan, circle, draw=black, inner sep=0.1cm, label=below:$a$] {};  
				\node (b) at (2.5,0)  [ball color=cyan, circle, draw=black, inner sep=0.1cm, label=below:$b$]{}; 
				\node (c) at (2,1) [ball color=cyan, circle, draw=black, inner sep=0.1cm, label=left :$c$]{};    
				\node (d) at (2,2) [ball color=cyan, circle, draw=black, inner sep=0.1cm, label=above :$d$]{};

				\draw (a) -- (c);
				\draw (b) -- (c); 
				\draw (c) -- (d);

			\node (e) at (6,0) [ball color=cyan, circle, draw=black, inner sep=0.1cm, label=below:$e$] {};  
			\node (f) at (5.5,1)  [ball color=cyan, circle, draw=black, inner sep=0.1cm, label=left:$f$]{}; 
			\node (g) at (6.5,1) [ball color=cyan, circle, draw=black, inner sep=0.1cm, label=right :$g$]{};    
			\node (h) at (6,2) [ball color=cyan, circle, draw=black, inner sep=0.1cm, label=above :$h$]{};

			\draw (e) -- (f);
			\draw (e) -- (g); 
			\draw (f) -- (h);
			\draw (g) -- (h);

		\end{tikzpicture}  
		
		 $Y$ \hspace{3mm} Example \ref{example 4} \hspace{2mm}  $Z$
			\end{center}
	
	\end{minipage}\\

	 Using above terminology with Theorem \ref{D}, we can say that if a point is responsible to end a class in a persistence diagram, then it is a cone point. In the next example, we notice that not all cone points ends a class in persistence diagram.

	 	\begin{center}
	 	\begin{tikzpicture}  
	 		[scale=.9,auto=center,every node/.style={circle,fill=blue!10}] 
	 		
	 		\node (a) at (1,0) [ball color=cyan, circle, draw=black, inner sep=0.1cm, label=below:$a$] {};  
	 		\node (b) at (2,0)  [ball color=cyan, circle, draw=black, inner sep=0.1cm, label=below:$b$]{}; 
	 		\node (c) at (3,0) [ball color=cyan, circle, draw=black, inner sep=0.1cm, label=below :$c$]{};
	 		\node (f) at (1.5,1) [ball color=cyan, circle, draw=black, inner sep=0.1cm, label=left :$f$]{};  
	 		\node (e) at (2.5,1) [ball color=cyan, circle, draw=black, inner sep=0.1cm, label=right :$\!\!e$]{};  
	 		\node (d) at (3.5,1) [ball color=cyan, circle, draw=black, inner sep=0.1cm, label=right :$d$]{};  
	 		\node (g) at (2,2) [ball color=cyan, circle, draw=black, inner sep=0.1cm, label=above :$g$]{};

	 		\draw (a) -- (f);
	 		\draw (b) -- (f); 
	 		\draw (f) -- (g);
	 		\draw (b) -- (e);
	 		\draw (e) -- (g); 
	 		\draw (c) -- (e); 
	 		\draw (c) -- (d); 
	 		\draw (d) -- (g);

	 		\node[ball color=black, circle, draw=black, inner sep=1pt,label=left:$a$] (a) at (8,1) {};
	 		\node[ball color=black, circle, draw=black, inner sep=1pt, label=right:$b$] (b) at (10.5,1.5) {};
	 		\node[ball color=black, circle, draw=black, inner sep=1pt,label=below:$c$] (c) at (9.8,-0.2) {};
	 		\node[ball color=black, circle, draw=black, inner sep=1pt,label=left:$d$] (d) at (8,0.2) {};
	 		\node[ball color=black, circle, draw=black, inner sep=1pt,label=right:$e$] (e) at (11,0.6) {};
	 		\node[ball color=black, circle, draw=black, inner sep=1pt,label=above:$f$] (f) at (9,2) {};
	 		\node[ball color=black, circle, draw=black, inner sep=1pt,label=below left:$g$] (g) at (9.8,0.8) {};

	 		\draw (a) -- (f) -- (g) -- (a);
	 		\draw (b) -- (f) -- (g) -- (b);
	 		\draw (d) -- (c) -- (g) -- (d);
	 		\draw (b) -- (g) -- (e) -- (b);
	 		\draw (c) -- (e) -- (g) -- (c);

	 		\fill[pattern=north east lines, pattern color=black] (8,1) -- (9.8,0.8) -- (9,2) -- cycle;
	 		\fill[pattern=north east lines, pattern color=black] (10.5,1.5) -- (9.8,0.8) -- (9,2) -- cycle;
	 		\fill[pattern=north east lines, pattern color=black] (8,0.2) -- (9.8,0.8) -- (9.8,-0.2) -- cycle;
	 		\fill[pattern=north east lines, pattern color=black] (10.5,1.5) -- (9.8,0.8) -- (11,0.6) -- cycle;
	 		\fill[pattern=north east lines, pattern color=black] (9.8,-0.2) -- (9.8,0.8) -- (11,0.6) -- cycle;
	 		
	 	\end{tikzpicture}  
	 \end{center}
	 \hspace{32mm}	$X$ \hspace{15mm} Example \ref{example3} \hspace{20mm} $\mathcal{S}(X)$\\
	 
	 \begin{example}\label{example3}
	 	{\normalfont
	 		Let $X=\{a,b,c,d,e,f,g\}$ be a finite $T_0$ space associated with a poset, whose Hasse diagram is shown in the above figure. Here, $dgm(V_a)=dgm(V_c)=dgm(V_d)=\{[0,\infty),[1,2),[1,3)\}$,  $dgm(V_b)=\{[0,\infty),[1,2),[1,2)\}$,\\ $dgm(V_e)=dgm(V_f)=\{[0,\infty),[1,2)\}$, and $dgm(V_g)=\{[0,\infty)\}$.
	 	}
	 \end{example}
	 
	 Notice that in above example, there are three cone point $e,f$ and $g$ in $X$. However, no class ends in any persistence diagram $dgm(V_x)$ for any $ x\in X$ due to appearance of $g$ in $X_x^t$. This shows that not all cone points necessarily give rise to the death of a class in the persistence diagram.\par 
	 Theorem \ref{D} states that the appearance of a point that is comparable to (at least) two minimal points is responsible for ending a class in the persistence diagram, provided that no such point comparable to those minimal points has appeared earlier. In Example \ref{example3}, $e$ and $f$ always appear before $g$. Therefore, $g$ is never responsible for ending any class.\par 
	 Next, we determine which cone points are responsible for the death of classes in the persistence diagram.
	 
	 \begin{remark}
	 	{\normalfont 
            Let $X$ be a finite $T_0$ space and $x \in X$. Let $\zeta_X^x$ denote the collection of cone points of $X$ that are responsible for ending a class in $dgm(V_x)$. From the furtherness matrix $\Psi(X)$, recall that if a row corresponding to a point contains exactly one zero, then that point is minimal \cite{Badra}. For a minimal point $a$ not comparable with $x$, there exists a class $[1, n_a^x) \in dgm(V_x)$, which corresponds to $a$. In this case, $\Psi(x,a)=1$. Thus, in the furtherness matrix, the presence of $1$ in the row of $x$ and column corresponding to a minimal point $a$ indicates the existence of a class in the persistence diagram $dgm(V_x)$.\par 
            Now, for each minimal point $a$, define $\Theta_a$ to be the collection of points that are comparable with $a$ as well as with at least one minimal point other than $a$. In other words, elements of $\Theta_a$ are precisely those points of $X$ whose corresponding rows have $0$ in the columns associated with $a$ and at least one minimal point other than $a$. So, for every minimal point $a$, we can obtain $\Theta_a$ from the furtherness matrix of $X$.\par
            If a class  $[1, n_a^x) \in dgm(V_x)$ corresponding to a minimal point  $a$ dies at level $t = n_a^x<\infty$, then  a point from $\Theta_a$ appears in  $X_x^t$. Hence, $\Theta_a$ must be non empty, and    $n_a^x=\min\{\Psi(x,y)~|~ y\in \Theta_a\}.$ And if $z \in \Theta_a$ satisfies $n_a^x=\min\limits_{y\in \Theta_a}\Psi(x,y)=\Psi(x,z)$, then $z\in \zeta^x_{X}$.\par 
              Note that if $\Theta_a\neq \emptyset,$ then it is not necessary  that $n_a^x<\infty$. However, this implication holds whenever X is connected.
	 		}
	 \end{remark}
	 
	 \noindent
	 \begin{minipage}{0.55\textwidth}
	 	 Notice that in Example \ref{example3}, we have three minimal points $a,b$ and $c$. From the furtherness matrix $\Psi(X)$, it is clear that $\Theta_a=\{f,g\}$, $\Theta_b=\{e,f,g\}$ and $\Theta_c= \{e,g\}$.\par 
	 	 And, $\zeta_X^a= \zeta_X^b= \zeta_X^c =\zeta_X^d=\{e,f\}$, $\zeta_X^e=\{f\}$, $\zeta_X^f=\{e\}$, and $\zeta_X^g=\emptyset$.
	 \end{minipage}
	 \hfill
	 \begin{minipage}{0.35\textwidth}
	 $\bordermatrix{& a & b & c & d & e & f & g \cr
	 		a & 0 & 1 & 1 & 2 & 3 & 2 & 6 \cr
	 		b & 1 & 0 & 1 & 2 & 2 & 2 & 6 \cr
	 		c & 1 & 1 & 0 & 1 & 2 & 3 & 6 \cr
	 		d & 1 & 1 & 0 & 0 & 2 & 3 & 5 \cr
	 		e & 1 & 0 & 0 & 1 & 0 & 2 & 4 \cr
	 		f & 0 & 0 & 1 & 2 & 2 & 0 & 4 \cr
	 		g & 0 & 0 & 0 & 0 & 0 & 0 & 0 }$
	 		
	 		\vspace{2mm} 
	 		
	  \hspace{6mm}  $\Psi(X)$ in Example \ref{example3}
	 \end{minipage}\\

	 We now define a relation on finite $T_0$ spaces using their minimal and cone points. 
	 \begin{remark}\label{M}{\normalfont
	 	Let $X$ be a finite $T_0$ space and $x,y\in X$. We define a relation on $X$ by setting $x\sim y$ if and only if the following conditions hold:
	 	\begin{itemize}
	 	\item[i.]  The number of minimal points of X that are not comparable to $x$ is equal to the number of minimal points that are not comparable to $y$.
	 	\item[ii.] The multisets $\{\Psi(x,t)~|~  t\in \zeta_X^x\}$ and  $\{\Psi(y,t)~|~  t\in \zeta_X^y\}$ are equal.
	 	\end{itemize} 
	 	
	 }
	 \end{remark}
	 
	 It is easy to verify that the relation defined in Remark \ref{M} is an equivalence relation on $X$. Thus, this relation partitions the space into equivalence classes.\par

	In Example \ref{example1}, we have $r\sim s\sim t\sim u$, in Example \ref{example3}, $a\sim c\sim d$ and $e\sim f$, while in Example \ref{example2}, each point of the space is related to every other point.\par 
	 The following result connects the equivalence relation defined in Remark \ref{M} with the persistence diagrams of the associated elements. 
   
   \begin{theorem}\label{M'}
    Let $X$ be a finite $T_0$ space and $\sim$ be a relation as defined in Remark \ref{M}. Then for $x,y\in X$, we have $x\sim y$ if and only if $dgm(V_x)=dgm(V_y)$.
   \end{theorem}
   
   \begin{proof}
      First, assume that $x\sim y$. By Remark \ref{F}, $|dgm(V_x)|=|dgm(V_y)|$. We have  $\{\Psi(x,t)~|~  t\in \zeta_X^x\} = \{\Psi(y,t)~|~  t\in \zeta_X^y\}$ as multisets. If $[1,n_a^x)\in dgm(V_x)$ such that $n_a^x<\infty$, where $a$ is a minimal point not comparable with $x$, then $n_a^x= \Psi(x,t)$, for some $t\in \zeta_X^x$.  Therefore, $n_a^x\in \{\Psi(x,t)~|~  t\in \zeta_X^x\} = \{\Psi(y,t)~|~  t\in \zeta_X^y\}$. Hence $n_a^x= \Psi(y,t') =n_b^y$, for some $t'\in  \zeta_X^y$ and some  minimal point $b$ not comparable with $y$. Consequently, $[1,n_a^x)= [1,n_b^y)\in dgm(V_y)$. Thus, every class of finite persistence in $dgm(V_x)$ is also present in $dgm(V_y)$. Note that the classes of infinite persistence in a persistence diagram are corresponding to the connected components of $\mathcal{S}(X)$. So, $dgm(V_x)\subseteq dgm(V_y)$. Similarly, we get $dgm(V_y)\subseteq dgm(V_x)$. Hence, $dgm(V_x) =  dgm(V_y)$.\par
      Now, let $dgm(V_x)=dgm(V_y)$. Again, by Remark \ref{F}, $x$ and $y$ have the same number of minmial points in $X$ that are  not comparable with them. Consider $t\in \zeta_X^x$, then $\Psi(x,t)= n_a^x$, for some minimal point $a$ not comparable with $x$. Thus, $[1, n_a^x)\in dgm(V_x) =dgm (V_y) \iff [1, n_a^x)=[1, n_b^y),$ for some minimal point $b$ not  comparable with $y$. Therefore, $\Psi(x,t)= n_a^x= n_a^y =\Psi(y,t')$, for some $t'\in \zeta_X^y$. Hence, The multisets $\{\Psi(x,t)~|~  t\in \zeta_X^x\}$ and  $\{\Psi(y,t)~|~  t\in \zeta_X^y\}$ are equal. Therefore, $x\sim y$.
       \end{proof}
	 
	 The above result shows that if $x\sim y$ in a finite $T_0$ space $X$ then their topological evolution within the space is same with respect to the furtherness.\par 
	 We now study the quotient space induced by this relation. Each equivalence class consists of elements of $X$ having the same persistence diagram. Hence, the quotient space is the collection of all such equivalence classes in which the topological evolutions of $\mathcal{S}(X)$ from its elements are same. That means, the quotient space is the collection of all possible topological evolutions of $\mathcal{S}(X)$ from its elements.\par 
	 
	For  a finite $T_0$ space $X$, its quotient space $X/$$\sim$ can be viewed as a preordered set. Using Proposition 2.18 of \cite{Badra} and the continuity of the quotient map $\pi$ from $X$ to $X/$$\sim$, we get 
	 
	 \begin{theorem}\label{KH}
	     Let $X$ be finite $T_0$ space and $x,y\in X$. Then $\Psi_X(x,y)=0 \implies \Psi_{X/\sim}([x],[y])=0.$
	 \end{theorem}
	 
	 It is equivalent to say that  $U_x\subset U_y$ implies $U_{[x]}\subset U_{[y]}$. Converse of Theorem \ref{KH} is not always true. In Example \ref{example1}, $ \Psi_{X/\sim}([q],[u])=0$ but $\Psi_X(q,u)\neq 0$.\par 
	 If $y\in U_x $, then $\Psi(y,x)=0 \implies \Psi_{X/\sim}([y],[x])=0$.  Thus, $\pi(y)=[y]\in U_{[x]}.$ So, we get

	 \begin{corollary}\label{G}
	 	Let $X$ be finite $T_0$ space and $x\in X$. Then $\pi(U_x)\subseteq U_{[x]}$, where $\pi:X\rightarrow X/$$\sim$ is the quotient map.
	 \end{corollary}

	 From Theorem \ref{A}, we can notice that if a finite $T_0$ space has only one minimal point then every element  of the space has the same persistence diagram. So, we have
	 
	 \begin{corollary}\label{C}
	 	Let $X$ be a finite $T_0$ space having just one minimal point and $\sim$ be an equivalence relation as defined in Remark \ref{M}. Then the quotient space $X/$$\sim$ is trivial.
	 \end{corollary}

	 In the next result, we show that if a chain in a finite $T_0$ space is attached to the rest of the space by a single point, then every element of the chain is related to that point.
	 
	 \begin{theorem}
	 	Let $X$ be a finite $T_0$ space and $a\leq x_1\leq\cdots \leq x_k$ be a chain of elements of $X$ such that $x_i$ is not comparable with any other point of $X$ outside this chain. Then $x_i\sim a, ~ 1\leq i\leq k$.  
	 \end{theorem}
	    
	 \begin{proof}
	 Note that the number of minimal points not comparable with $x_i$ is same as number of minimal points not comparable with $a$. Next, we observe that for every cone point $y\in X, \Psi(x_i,y)=\Psi(a,y)$. Since $x_i$ is not comparable with $y$, we have $x_i\notin U_y$. So, we get $U_y\backslash U_{x_i}= U_y\backslash (U_{a} \cup \{x_1,x_2,\cdots, x_i\})$. Therefore, $|U_y\backslash U_{x_i}| = |U_y\backslash U_{a}|$. This gives  $\Psi(x_i,y)=\Psi(a,y).$ Hence, $\zeta_X^{x_i}=\zeta_X^a$. Thus the multisets $\{\Psi(x_i,y)~|~ y\in \zeta_X^{x_i}\}$ and $\{\Psi(a,y)~|~ y\in\zeta_X^a \}$ are equal. Therefore, $x_i\sim a, ~\forall~ 1\leq i\leq k$.  
	 \end{proof}
	 
	 	The following result provides insight into the structure of the equivalence classes induced by the relation defined in Remark \ref{M}.

\begin{theorem}\label{CH}
	Let $X$ be a finite $T_0$ space and $x,y,z\in X$, such that $x\leq y\leq z$ and $x\sim z$. Then $x\sim y\sim z$.
\end{theorem}

\begin{proof}
	As $x\sim z$, we get $|dgm(V_x)|=|dgm(V_z)|$, and  as multisets $\{\Psi(x,t)~|~ t\in \zeta_X^{x}\} = \{\Psi(z,t)~|~ t\in\zeta_X^z \}$. As $x\leq y\leq z$, by Theorem \ref{K}, $|dgm(V_z)|\leq |dgm(V_y)|\leq |dgm(V_x)|$. Therefore, $|dgm(V_z)| = |dgm(V_y)|= |dgm(V_x)|$.\par 
	Let $\mathcal{A}_x, \mathcal{A}_y$ and $\mathcal{A}_z$ be the collections of minimal points of $X$ not comparable with $x,y$ and $z$, respectively. So, $|\mathcal{A}_x|=|\mathcal{A}_y|=|\mathcal{A}_z|$. Clearly,  $\mathcal{A}_z\subseteq \mathcal{A}_y \subseteq \mathcal{A}_x$. This gives that $\mathcal{A}_z= \mathcal{A}_y =\mathcal{A}_x$. We know that $\zeta_X^x$ is the collection of those cone points of $X$, which are responsible to end a class in $dgm(V_x)$. For some $a\in \mathcal{A}_x$, if $n_a^x<\infty$, then $n_a^x\in \{\Psi(x,t)~|~ t\in \zeta_X^x\}$ and $[1,n_a^x)\in dgm(V_x)$. Therefore, $\{n_a^x~|~ a\in \mathcal{A}_x\}= \{\Psi(x,t)~|~ t\in \zeta_X^x\} ~\cup~ \{\infty$ with multiplicity $k-1\}$, where $k$ is the number of connected components of $\mathcal{S}(X)$. As $x\sim z$, $\{\Psi(x,t)~|~ t\in \zeta_X^x\}=\{\Psi(z,t)~|~ t\in \zeta_X^z\} \implies \{n_a^x~|~ a\in \mathcal{A}_x\}=\{n_a^z~|~ a\in \mathcal{A}_z\}$. By Theorem \ref{E},  for every $a\in \mathcal{A}_z$, we have $n_a^x\leq n_a^y\leq n_a^z$.  Consequently, $\{n_a^x~|~ a\in \mathcal{A}_x\}=\{n_a^y~|~ a\in \mathcal{A}_y\}=\{n_a^z~|~ a\in \mathcal{A}_z\}$. Hence, $\{\Psi(x,t)~|~ t\in \zeta_X^x\}=\{\Psi(y,t)~|~ t\in \zeta_X^y\}=\{\Psi(z,t)~|~ t\in \zeta_X^z\}$. Thus, $x\sim y\sim z$.
\end{proof}

Next, we  generalize above result for arbitrary length of a chain in finite $T_0$ space.

\begin{corollary}\label{CHH}
       Let $X$ be a finite $T_0$ space and  $x_1\leq x_2\leq \cdots \leq x_k$ be a chain of elements of $X$ such that $x_1\sim x_k$. Then $x_1\sim x_2\sim \cdots \sim x_k$.
\end{corollary}

    The above result shows that the equivalence classes induced by the relation are convex with respect to the partial order of the space. Indeed, if two comparable elements belong to the same equivalence class, then every element lying between them also belongs to that class.
	  
	\section{\textbf{\large{Extending the Construction to Finite Spaces}}}\label{6}
	
	In the previous sections, we associated a persistence module to each element of a finite $T_0$ space, studied the structure of the corresponding persistence diagrams, and classified elements of the space according to their associated persistence diagrams.
	In this section, using the furtherness preserving property \cite{Badra}, we extend these constructions and results from finite $T_0$ spaces to arbitrary finite spaces.\par 
	Let $X$ be a finite space and $X_0$ be the quotient space $X/$$\sim_0$, where $x\sim_0 y$ if $U_x=U_y$. By Proposition 1.3.1 of \cite{Barmak}, $X_0$ is a finite $T_0$ space and homotopically equivalent to $X$.  Further, by Theorem 2.22 of \cite{Badra}, we get $\Psi([x],[y])=\Psi(x,y)$ for all $x,y\in X$, where $\Psi$ is the furtherness in respective spaces as introduced in \cite{Badra}. Thus, the quotient map $q:X\rightarrow X_0$ is a furtherness preserving map. We associate a persistence module with each element of $X$ using the persistence module of the corresponding element in the quotient space $X_0$.\par
	For $x\in X$ and $k\geq 0$, let $X_x^k=\{y\in X~|~ \Psi(x,y)\leq k\}$. Note that if $y\sim_0 x$, then $y\in X_x^0$, and hence $[x]\subseteq X_x^0$. Similarly, in the filtration $\{X_x^k~|~ k\geq 0\}$, whenever a point  $y$ enters $X_x^t$, every element of the equivalence class $[y]$ also enters $X_x^t$, at the same level $t$.  Consequently, every element of $[x]$ determines the same filtration.
    Note that, if  $y$ enters $X_x^t$, then $[y]$ enters $X^t_{0_{[x]}}$, where $X^t_{0_{[x]}}=\{[y]\in X_0 ~|~ \Psi_{X_0}([x],[y])\leq t\}$. In fact, the topolgical evolution of $X$ from $x$  is identical to the topolgical evolution of $X_0$ from $[x]$ with respect to their respective furthernesses. Since, $X_0$ is a finite $T_0$ space, we have an associated order complex $\mathcal{S}(X_0)$. By the construction in Section \ref{3}, we have a persistence module $V_{[x]}: \mathbb{R} \rightarrow \mathbf{Vect}_{\mathbb{K}}$ such that $V_{[x]}(a)=H_0(F_{[x]}(a),\mathbb{K}), ~\forall~ [x]\in X_0$. Hence for every element $x$ of a finite space  $X$, we can associate a persistence module $V_{[x]}$.\par 
    
     In the following results, $X_0$ denotes the quotient space of a finite space $X$ as constructed above.

	\begin{definition}\label{GY}
		Let $X$ be a finite space and  $x\in X$. The Persistence diagram $dgm(V_x)$ assiciated with $x$ is defined as the persistence diagram $dgm(V_{[x]})$ associated with $[x]$ in the quotient space $X_0$. 
	\end{definition}
	
	 The following results describe the structure of the persistence diagrams associated with the persistence modules of elements of a finite space $X$. These results explain when a class appears and disappears in the persistence diagram of the persistence module $V_x$ associated with $x\in X$.

	\begin{theorem}\label{A'}
	Let $X$ be a finite space and $x\in X$. Then
	\begin{itemize}
		\item[$(i)$] $[0,\infty)\in dgm(V_{x})$ with multiplicity one.
		\item[$(ii)$] If $X_0$ has $m$ minimal points which are not comparable with $[x]$, then $m$ new classes appear at level 1.
		\item[$(iii)$] No new class appear after level 1.
	\end{itemize}
	
\end{theorem}

\begin{theorem}\label{D'}
	Let $X$ be a finite space and $x\in X$. Let $[b]$ be a minimal point of $X_0$ not comparable with $[x]$. If a class born at level 1 in $dgm(V_{x})$ corresponding to $[b]$, dies at level $n_{[b]}<\infty,$ then there is an element in $X_{[x]}^{n_{[b]}}$ which is comparable with both $[b]$ and another minimal point of $X_0$. Moreover, if no such element exists then $n_{[b]}=\infty$.
\end{theorem}

\begin{remark}\label{M''}{\normalfont
		Let $X_0$ be the quotient space of a finite space $X$ and $x,y\in X$. We define a relation on $X$ by setting $x\sim' y$ if and only if $[x]$ and  $[y]$ are related by the relation defined in Remark \ref{M}. By Theorem \ref{M'}, it is equivalent to say that $x\sim' y \iff dgm(V_{[x]})=dgm(V_{[y]})$.\par 
		Notice that the quotient map $p:X\rightarrow X/$$\sim'$ factors as $p=\pi\circ q$, where $q: X\rightarrow X_0$ is the quotient map, and $\pi: X_0\rightarrow X_0/$$\sim$ is the quotient map induced by the equivalence relation $\sim$ defined in Remark \ref{M}. Thus, the quotient space $X/$$\sim'$ can be obtained by first passing to the $T_0$-quotient $X_0= X/$$\sim_0$, and then quotienting by $\sim$. Therefore, $X/$$\sim'$ and $X_0/$$\sim$ are topologically identical, as both quotient spaces have the same cardinality and the same topology. Consequently, $X/$$\sim'$ $= X_0/$$\sim$. Thus, the quotient space $X/$$\sim'$  is completely determined by the homotopically equivalent $T_0$ quotient space $X_0$. Furthermore, if $X$ is a $T_0$ space, then the equivalence relation $\sim'$ coincides with $\sim$.}
\end{remark}

The remaining constructions and results can be generalized from finite $T_0$ spaces to arbitrary finite spaces in an analogous manner. In the next result, we study the relationship between the persistence diagram associated with an element and the persistence diagram associated with its image under a furtherness preserving map.\par 
 Let $f: X\rightarrow Y$ be a furtherness preserving  map between two finite $T_0$ spaces. If $f$ is not injective, then $f(x_1)=f(x_2)$ for some $x_1 \neq x_2$ in  $X$. As $f$ is furtherness preserving map, $\Psi_X(x_1,x_2)= \Psi_X(x_2,x_1)=0 \implies U_{x_1}=U_{x_2}$, a contradiction, as $X$ is $T_0$. So, we have

\begin{lemma}\label{T}
     A furtherness preserving  map between two finite $T_0$ spaces is injective. 
\end{lemma}

    Let $f: X\rightarrow Y$ be a furtherness preserving  map between two finite spaces $X$  and $Y$. For $x\in X$, if   $x_0\in [x]$ then  $\Psi_X(x,x_0)= \Psi_X(x_0,x)=0 \implies \Psi_Y(f(x),f(x_0))= \Psi_Y(f(x_0),f(x))=0\implies U_{f(x)}=U_{f(x_0)}.$ This gives $x\sim_0 x_0 \implies f(x)\sim_0 f(x_0)$. So, the map $\overline{f}: X_0\rightarrow Y_0$  between their respective quotient spaces defined by $\overline{f}([x])=[f(x)]$, is a well defined map. As the quotient map $q:X\rightarrow X_0$  is a furtherness preserving map, we have
    
    \begin{lemma}\label{TH}
    	Let $f: X\rightarrow Y$ be a furtherness preserving  map between two finite spaces $X$  and $Y$. Then $f$ induces a furtherness preserving map $\overline{f}: X_0\rightarrow Y_0$ between their respective quotient spaces $X_0$ and $Y_0$.
    \end{lemma}
    
   By Lemma \ref{T} and Lemma \ref{TH}, if $f: X\rightarrow Y$ is a furtherness preserving surjective map between two finite spaces, then the induced map $\overline{f}: X_0\rightarrow Y_0$ is a furtherness preserving bijective map. By Proposition 2.18 of \cite{Badra}, both $\overline{f}$ and its inverse are continuous maps. Therefore, $\overline{f}: X_0\rightarrow Y_0$ is a homeomorphism. Thus, $dgm(V_{[x]})= dgm(V_{[f(x)]})$, for $[x]\in X_0$. So, we have

\begin{theorem}\label{KSH}
	 Let $f: X\rightarrow Y$ be a furtherness preserving surjective map between two finite spaces. Then $dgm(V_x)= dgm(V_{f(x)}), ~\forall~ x\in X$.
\end{theorem}

	\section{\textbf{\large{Bottleneck Stability in Finite Spaces}}}\label{7}
	
In this section, we establish a stability result for the persistence diagrams associated with elements of a finite space. We show that the bottleneck distance between the persistence diagrams of two elements is bounded above by the maximum furtherness from one another.\par 

	For any finite space $X$ and $x\in X$, there is a monotonic map $f_{x}:\mathcal{S}(X_0)\rightarrow \mathbb{R}$ defined as  $f_x(\sigma )=\max\{\Psi(x,a_i)~|~ 0\leq i\leq n\},$ for $\sigma= \langle a_0,a_1,...,a_n\rangle	\in \mathcal{S}(X_0).$ By the  construction of Section \ref{3}, $f_x$ induces the filtration $F_x$, and hence the associated persistence module $V_x$. Therefore, the $0^{th}$ persistence diagram of $f_x$ is $dgm(V_x)$.\par 
	Note that, for $x,y\in X$, we get $f_x(\sigma)=\Psi(x,a_l)$ and $f_y(\sigma)=\Psi(y,a_m)$, for some $0\leq l,m\leq n$. As furtherness satisfies triangle inequality \cite{Badra}, and $\Psi(y,a_l)\leq \Psi(y,a_m) \implies \Psi(x,a_l)- \Psi(y,a_m)\leq  \Psi(x,y)$. Similarly, $\Psi(x,a_m)\leq \Psi(x,a_l) \implies \Psi(y,a_m)- \Psi(x,a_l)\leq \Psi(y,x)$. Therefore, $|f_x(\sigma)-f_y(\sigma)|\leq  \max\{\Psi(x,y),\Psi(y,x)\}$. Hence $||f_x-f_y||_{\infty}= \sup\limits_{\sigma\in \mathcal{S}(X_0)}|f_x(\sigma)-f_y(\sigma)|\leq \max\{\Psi(x,y),\Psi(y,x)\}$. By Proposition \ref{S}, we get $d_B(dgm(V_x),dgm(V_y))\leq \max\{\Psi(x,y),\Psi(y,x)\}$.\par 
	 So, we have the following stability result for finite topological spaces.

	\begin{theorem}[Bottleneck Stability]\label{SH}
		Let $X$ be a finite space and $x,y\in X$. Then the bottleneck distance between the persistence diagrams  associated with $x$ and $y$ is bounded above by the maximum of their mutual furtherness; that is,  $d_B(dgm(V_x),dgm(V_y)) \leq \max\{\Psi(x,y),\Psi(y,x)\}$.
	\end{theorem}
	
The above theorem shows that small variations in the furtherness between elements lead to only limited changes in their associated persistence diagrams. Thus, the persistent homology arising from the furtherness filtration is stable under perturbations of the underlying furtherness function. This theorem also provides a practical estimate for comparing persistence diagrams.

\end{document}